\documentclass[11pt]{article}
\usepackage{amsmath,amsthm,amsfonts,amssymb,enumerate,graphicx,float,calc}
\usepackage{algorithm}
\usepackage{algorithmic}   
 \floatname{algorithm}{Algorithm} 
\usepackage[colorlinks=true,citecolor=black,linkcolor=black,urlcolor=blue]{hyperref}

\usepackage[lmargin=31mm,rmargin=31mm,bmargin=31mm,tmargin=31mm]{geometry}
\usepackage[numbers,sort&compress]{natbib}
\DeclareMathOperator{\depth}{depth}
\DeclareMathOperator{\blah}{col}

\renewcommand{\baselinestretch}{1.088}
\renewcommand{\thefootnote}{\fnsymbol{footnote}}	
\allowdisplaybreaks

\newcommand{\arXiv}[1]{arXiv:\,\href{http://arxiv.org/abs/#1}{#1}}
\newcommand{\msn}[1]{MR:\,\href{http://www.ams.org/mathscinet-getitem?mr=MR#1}{#1}}
\newcommand{\MSN}[2]{MR:\,\href{http://www.ams.org/mathscinet-getitem?mr=MR#1}{#1}}
\newcommand{\doi}[1]{doi:\,\href{http://dx.doi.org/#1}{#1}}

\theoremstyle{plain}
\newtheorem{theorem}{Theorem}
\newtheorem{lemma}[theorem]{Lemma}
\newtheorem{corollary}[theorem]{Corollary}

\theoremstyle{definition}

\begin{document}

\author{ Vida Dujmovi{\'c}\,\footnotemark[2] \qquad Fabrizio
  Frati\,\footnotemark[3] \qquad Gwena\"el Joret\,\footnotemark[4]
  \qquad David~R.~Wood\,\footnotemark[5]}

\footnotetext[2]{School of Computer Science, Carleton University,
  Ottawa, Canada (\texttt{vida@scs.carleton.ca}). Supported by NSERC
  and an Endeavour Fellowship from the Australian Government.}

\footnotetext[3]{School of Information Technologies, The University of
  Sydney, Sydney, Australia (\texttt{brillo@it.usyd.edu.au}).}

\footnotetext[4]{D\'epartement d'Informatique, Universit\'e Libre de
  Bruxelles, Brussels, Belgium
  (\texttt{gjoret@ulb.ac.be}). Postdoctoral Researcher of the Fonds
  National de la Recherche Scientifique (F.R.S.--FNRS). Supported by
  an Endeavour Fellowship from the Australian Government.}

\footnotetext[5]{Department of Mathematics and Statistics, The
  University of Melbourne, Melbourne, Australia
  (\texttt{woodd@unimelb.edu.au}). Supported by a QEII Fellowship and
  a Discovery Project from the Australian Research Council.}

\title{\boldmath\bf Nonrepetitive Colourings of Planar Graphs\\ with $O(\log
  n)$ Colours}
\maketitle


\begin{abstract}
  A vertex colouring of a graph is \emph{nonrepetitive} if there is no
  path for which the first half of the path is assigned the same
  sequence of colours as the second half. The \emph{nonrepetitive
    chromatic number} of a graph $G$ is the minimum integer $k$ such
  that $G$ has a nonrepetitive $k$-colouring. Whether planar graphs
  have bounded nonrepetitive chromatic number is one of the most
  important open problems in the field. Despite this, the best known
  upper bound is $O(\sqrt{n})$ for $n$-vertex planar graphs. We prove
  a $O(\log n)$ upper bound.
\end{abstract}

\section{Introduction}

\renewcommand{\thefootnote}{\arabic{footnote}}


A vertex colouring of a graph is \emph{nonrepetitive} if there is no
path for which the first half of the path is assigned the same
sequence of colours as the second half.  More precisely, a
$k$-\emph{colouring} of a graph $G$ is a function $\psi$ that assigns
one of $k$ colours to each vertex of $G$.  A path
$(v_1,v_2,\dots,v_{2t})$ of even order in $G$ is \emph{repetitively} coloured by
$\psi$ if $\psi(v_i)=\psi(v_{t+i})$ for all
$i\in[1,t]:=\{1,2,\dots,t\}$. A colouring $\psi$ of $G$ is
\emph{nonrepetitive} if no path of $G$ is repetitively coloured by
$\psi$. Observe that a nonrepetitive colouring is \emph{proper}, in
the sense that adjacent vertices are coloured differently. The
\emph{nonrepetitive chromatic number} $\pi(G)$ is the minimum integer
$k$ such that $G$ admits a nonrepetitive $k$-colouring.

The seminal result in this field is by \citet{Thue06}, who in 1906
proved
that every path is nonrepetitively 3-colourable. Nonrepetitive
colourings have recently been widely studied
\citep{HJ-DM11,Currie-EJC02,HJSS,PZ09,BreakingRhythm,BC11,BaratWood-EJC08,BV-NonRepVertex07,BK-AC04,GKM11,BGKNP-NonRepTree-DM07,Grytczuk,Currie-TCS05,Manin,Rampersad,NOW,KP-DM08,FOOZ,Pegden11,BV-NonRepEdge08,CG-ENDM07,Grytczuk-EJC02,AGHR-RSA02,GPZ,
  JS09,DS09}; see the surveys
\citep{Grytczuk,Grytczuk-DM08,Gryczuk-IJMMS07,CSZ}.  A number of graph
classes are known to have bounded nonrepetitive chromatic number. In
particular, trees are nonrepetitively 4-colourable
\citep{BGKNP-NonRepTree-DM07,KP-DM08}, outerplanar graphs are
nonrepetitively $12$-colourable \citep{KP-DM08,BV-NonRepVertex07}, and
more generally, every graph with treewidth $k$ is nonrepetitively
$4^k$-colourable \citep{KP-DM08}. Graphs with maximum degree $\Delta$
are nonrepetitively $O(\Delta^2)$-colourable
\citep{AGHR-RSA02,Grytczuk,Gryczuk-IJMMS07,HJ-DM11}.

Perhaps the most important open problem in the field of nonrepetitive
colourings is whether planar graphs have bounded nonrepetitive
chromatic number. This question, first asked by \citet{AGHR-RSA02},
has since been mentioned by numerous authors
\citep{Grytczuk,Grytczuk-DM08,Gryczuk-IJMMS07,KP-DM08,GPZ,BV-NonRepVertex07,BreakingRhythm,HJ-DM11,Grytczuk-EJC02,NOW,CG-ENDM07,DS09}. It
is widely known that $\pi(G)\in O(\sqrt{n})$ for $n$-vertex planar
graphs\footnote{One can prove this bound using a naive application of
  the Lipton-Tarjan planar separator theorem.}, and this is the best
known upper bound. The best known lower
bound is $11$, due to Pascal
  Ochem; see Appendix~\ref{LowerBound}. Here we prove a
logarithmic upper bound.

\begin{theorem}
  \label{NonRepPlanar}
  For every planar graph $G$ with $n$ vertices,$$\pi(G)\leq
  8(1+\log_{3/2}n)\enspace.$$
\end{theorem}

As a secondary contribution, we solve the above open problem when restricted to paths of bounded length.

\begin{theorem}
  \label{LocalPlanar}
  There is a constant $c$ such that, for every integer $k\geq 1$,
  every planar graph $G$ is $c^{k^2}$-colourable such that $G$
  contains no repetitively coloured path of order at most $2k$.
\end{theorem}

Note that the case $k=2$ of Theorem~\ref{LocalPlanar} corresponds to
so-called \emph{star colourings}; that is, proper colourings with no
2-coloured $P_4$; see
\citep{Albertson-EJC04,FRR-JGT04,NesOdM-03,Wood-DMTCS05}. \citet{Albertson-EJC04}
proved that every planar graph is star colourable with 20 colours.

\section{Proof of Theorem~\ref{NonRepPlanar}}

A \emph{layering} of a graph $G$ is a partition $V_0,V_1,\dots,V_p$ of
$V(G)$ such that for every edge $vw\in E(G)$, if $v\in V_i$ and $w\in
V_j$ then $|i-j|\leq 1$. Each set $V_i$ is called a \emph{layer}.  The
following lemma by \citet{KP-DM08} will be useful.



\begin{lemma}[\cite{KP-DM08}] \label{Pattern} For every layering of a
  graph $G$, there is a (not necessarily proper) 4-colouring of $G$
  such
  that 
  for every repetitively coloured path $(v_1,v_2,\dots,v_{2t})$, the
  subpaths $(v_1,v_2,\dots,v_{t})$ and
  $(v_{t+1},v_{t+2},\dots,v_{2t})$ have the same layer pattern.
\end{lemma}

A \emph{separation} of a graph $G$ is a pair $(G_1,G_2)$ of subgraphs
of $G$, such that $G=G_1\cup G_2$. In particular, there is no edge of
$G$ between $V(G_1)-V(G_2)$ and $V(G_2)-V(G_1)$.

\begin{lemma}
  \label{General}
  Fix $\epsilon\in(0,1)$ and $c\geq 1$. Let $G$ be a graph with $n$
  vertices. Fix a layering $V_0,V_1,\dots,V_p$ of $G$. Assume that,
  for every set $B\subseteq V(G)$, there is a separation $(G_1,G_2)$
  of $G$ such that:

  \vspace*{-3ex}
  \begin{itemize}
  \item each layer $V_i$ contains at most $c$ vertices in $V(G_1)\cap
    V(G_2)\cap B$, and
  \item both $V(G_1)-V(G_2)$ and $V(G_2)-V(G_1)$ contain at most
    $(1-\epsilon)|B|$ vertices in $B$.
  \end{itemize}

  \vspace*{-3ex} Then $\pi(G)\leq 4c(1+\log_{1/(1-\epsilon)}n)$.
\end{lemma}

\begin{proof}
  Run the following recursive algorithm \textsc{Compute}$(V(G),1)$.

\begin{center}
  \framebox{
    \begin{minipage}{\textwidth-5mm}
      \textsc{Compute}$(B,d)$

      \begin{enumerate}
      \item If $B=\emptyset$ then exit.
      \item Let $(G_1,G_2)$ be a separation of $G$ such that each
        layer $V_i$ contains at most $c$ vertices in $V(G_1)\cap
        V(G_2)\cap B$, and both $V(G_1)-V(G_2)$ and $V(G_2)-V(G_1)$
        contain at most $(1-\epsilon)|B|$ vertices in $B$.
      \item Let $\depth(v):=d$ for each vertex $v\in V(G_1)\cap
        V(G_2)\cap B$.
      \item For $i\in[1,p]$, injectively label the vertices in
        $V_i\cap V(G_1)\cap V(G_2)\cap B$ by $1,2,\dots,c$.\\ Let
        $\textup{label}(v)$ be the label assigned to each vertex $v\in
        V_i\cap V(G_1)\cap V(G_2)\cap B$.
      \item\textsc{Compute}$((V(G_1)-V(G_2))\cap B,d+1)$
      \item \textsc{Compute}$((V(G_2)-V(G_1))\cap B,d+1)$ \smallskip
      \end{enumerate}

    \end{minipage}}
\end{center}

The recursive application of \textsc{Compute} determines a rooted
binary tree $T$, where each node of $T$ corresponds to one call to
\textsc{Compute}. Associate each vertex whose depth and label is
computed in a particular call to \textsc{Compute} with the
corresponding node of $T$. (Observe that the depth and label of each
vertex is determined exactly once.)

Colour each vertex $v$ by $(\blah(v),\depth(v),\textup{label}(v))$,
where $\blah$ is the 4-colouring from Lemma~\ref{Pattern}. Suppose on
the contrary that $(v_1,v_2,\dots,v_{2t})$ is a repetitively coloured
path in $G$. By Lemma~\ref{Pattern}, $(v_1,v_2,\dots,v_{t})$ and
$(v_{t+1},v_{t+2},\dots,v_{2t})$ have the same layer pattern.  In
addition, $\depth(v_i)=\depth(v_{t+i})$ and
$\textup{label}(v_{i})=\textup{label}(v_{t+i})$ for all $i\in[1,t]$.
Let $v_i$ and $v_{t+i}$ be vertices in this path with minimum depth.
Since $v_i$ and $v_{t+i}$ are in the same layer and have the same
label, these two vertices were not labelled at the same step of the
algorithm.  Let $x$ and $y$ be the two nodes of $T$ respectively
associated with $v_i$ and $v_{t+i}$. Let $z$ be the least common
ancestor of $x$ and $y$ in $T$. Say node $z$ corresponds to call
\textsc{Compute}$(B,d)$. Thus $v_i$ and $v_{t+i}$ are in $B$ (since if
a vertex $v$ is in $B$ in the call to $\textsc{Compute}$ associated
with some node $q$ of $T$, then $v$ is in $B$ in the call to
$\textsc{Compute}$ associated with each ancestor of $q$ in $T$).  Let
$(G_1,G_2)$ be the separation in \textsc{Compute}$(B,d)$. Since
$\depth(v_i)=\depth(v_{t+i})>d$, neither $v_i$ nor $v_{t+i}$ are in
$V(G_1)\cap V(G_2)$. Since $z$ is the least common ancestor of $x$ and
$y$, without loss of generality, $v_i\in V(G_1)-V(G_2)$ and
$v_{t+i}\in V(G_2)-V(G_1)$.
Thus some vertex $v_j$ in the subpath
$(v_{i+1},v_{i+2},\dots,v_{t+i-1})$ is in $V(G_1)\cap V(G_2)$. If
$v_j\in B$ then $\depth(v_j)=d$. If $v_j\not\in B$ then
$\depth(v_j)<d$. In both cases,
$\depth(v_j)<\depth(v_i)=\depth(v_{t+i})$, which contradicts the
choice of $v_i$ and $v_{t+i}$. Hence there is no repetitively coloured
path in $G$.

Observe that the maximum depth is at most
$1+\log_{1/(1-\epsilon)}n$. Therefore the number of colours is at most
$4c(1+\log_{1/(1-\epsilon)}n)$.
\end{proof}


We now prove that the condition in Lemma~\ref{General} holds for
\emph{plane triangulations}; that is, embedded planar graphs in which
every face is a triangle. If $r$ is a vertex of a connected graph $G$
and $V_i$ is the set of vertices in $G$ at distance $i$ from $r$, then
$V_0,V_1,V_2,\dots$ is a layering of $G$, called the \emph{layering
  starting} at $r$. Observe that for each vertex $v\in V_i$ there is a
$vr$-path that contains exactly one vertex from each layer
$V_0,V_1,\dots,V_i$; we call this a \emph{monotone} path.

\begin{lemma}
  \label{PlanarSeparation}
  Let $r$ be a vertex in a plane triangulation $G$.  Let
  $V_0,V_1,\dots,V_p$ be the layering of $G$ starting at $r$.  For
  every set $B\subseteq V(G)$, there is a separation $(G_1,G_2)$ of
  $G$ such that:

  \vspace*{-3ex} \begin{itemize}
  \item each layer $V_i$ contains at most two vertices in $V(G_1)\cap
    V(G_2)\cap B$,
  \item both $V(G_1)-V(G_2)$ and $V(G_2)-V(G_1)$ contain at most
    $\frac{2}{3}|B|$ vertices in $B$.
  \end{itemize}
\end{lemma}

\begin{proof}
  If $|B|\leq2$ then $G_1:=G_2:=G$ satisfy the claim. Now assume that
  $|B|\geq3$.  A \emph{lollipop} $S$ of \emph{height} $k$ is a walk in
  $G$ such that: \vspace*{-2ex}
  \begin{itemize}
  \item either
    $S=(u_0,u_1,\dots,u_{k-1},u_k,v_k,v_{k-1},\dots,v_1,v_0)$ as in
    Figure~\ref{Lollipop}(a), \\or
    $S=(u_0,u_1,\dots,u_{k-1},u_k,u_{k+1},v_k,v_{k-1},\dots,v_1,v_0)$
    as in
    Figure~\ref{Lollipop}(b),\\
    where $u_i,v_i\in V_i$ for each $i\in[1,k]$, and $u_{k+1}\in
    V_{k+1}$;
  \item $u_0=v_0=r$ and $u_k\neq v_k$; and
  \item if $u_i=v_i$ for some $i\in[1,k-1]$, then $u_j=v_j$ for each
    $j\in[0,i]$.
  \end{itemize}

  \begin{figure}[htb]
    \begin{center}
      \includegraphics[width=150mm]{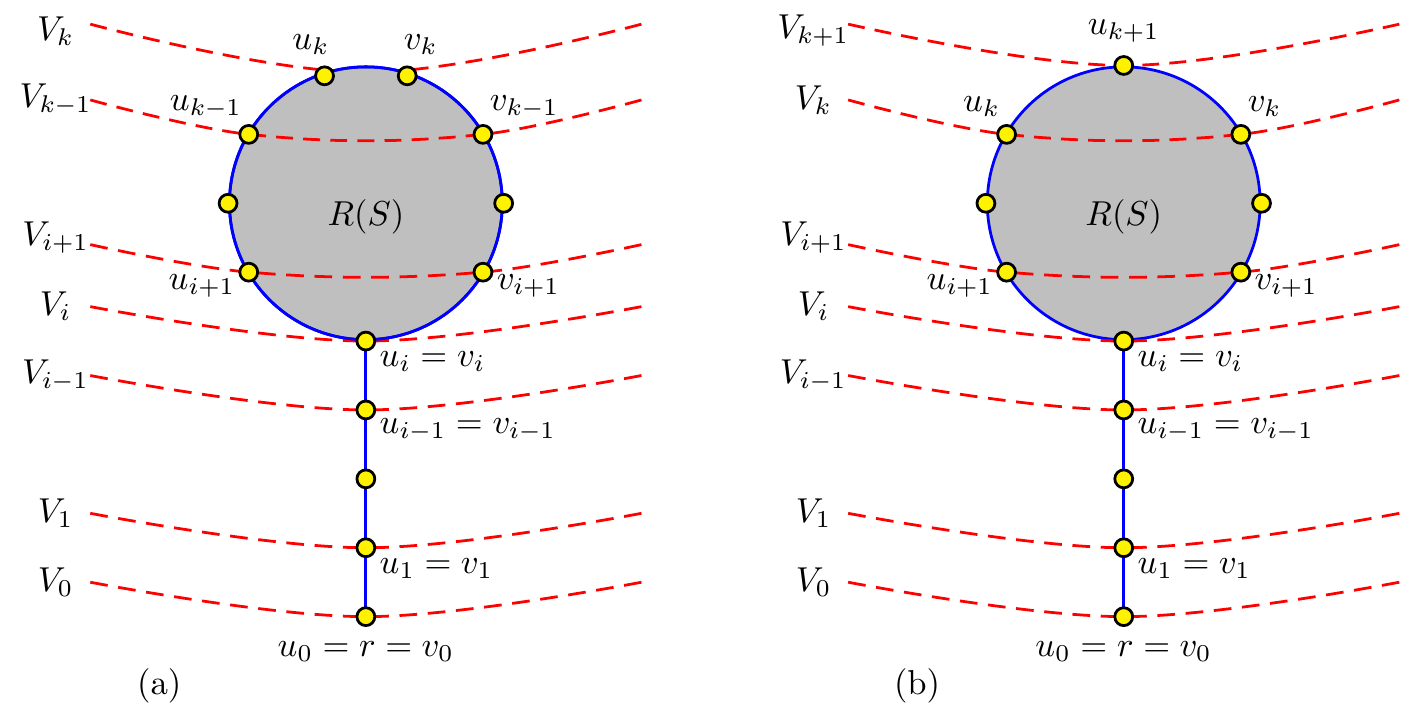}
      \caption{Two lollipops of height $k$. Note that the layers might
        have a more complicated structure than that shown
        here.\label{Lollipop}}
    \end{center}
  \end{figure}





  Consider a lollipop $S$. We define vertices to the right and left of
  $S$ as follows.  Let $i\geq 0$ be the maximum index for which
  $u_i=v_i$.  Let $C_S$ be the cycle obtained from $S$ by removing
  $u_0,u_1,\dots,u_{i-1}$ ($=v_0,v_1,\dots,v_{i-1}$). Then $w$ is
  \emph{to the right} of $S$ if it is to the right of $C_S$ when
  traversing $C_S$ so that vertex $u_{i+1}$ is visited immediately
  after vertex $u_i$. A vertex $w$ of $G$ is \emph{to the left} of $S$
  if it is neither to the right of $S$ nor a vertex of $S$. For the
  given set $B\subseteq V(G)$, let $R_B(S)$ and $L_B(S)$ be the sets
  of vertices in $B$ to the right and left of $S$, respectively. Let
  $r_B(S):=|R_B(S)|$ and $\ell_B(S):=|L_B(S)|$. We drop the subscript
  $B$ when $B=V(G)$.  By the Jordan Curve Theorem, $R_B(S)$ and
  $L_B(S)$ are disjoint.  Note that the reverse sequence
  $\overleftarrow{S}$ is also a lollipop, and
  $L_B(S)=R_B(\overleftarrow{S})$ and $R_B(S)=L_B(\overleftarrow{S})$.


 Let $S$ be a lollipop such that: \vspace*{-2ex}
  \begin{enumerate}[(1)]
  \item $r_B(S)\leq \frac{2}{3}|B|$;
  \item subject to (1), $r_B(S)$ is maximum; and
  \item subject to (1) and (2), $\ell(S)$ is minimum.
  \item subject to (1), (2) and (3), $r(S)$ is maximum.
  \end{enumerate}

  A lollipop satisfying (1) always exists, since if $(r,u_1,v_1)$ is a
  face in clockwise order, then $S:=(u_0=r,u_1,v_1,v_0=r)$ is a
  lollipop of height $1$ with $r_B(S)=0\leq\frac{2}{3}|B|$. 

Say $S$  has height $k$.  Suppose, for the sake of contradiction, that
  $\ell_B(S)>\frac{2}{3}|B|$. Thus $R_B(S)<\frac{1}{3}|B|$.  We
  distinguish the following cases:

  \emph{Case~1.}
  $S=(u_0,u_1,\dots,u_{k-1},u_k,v_k,v_{k-1},\dots,v_1,v_0)$: Let $w$
  be the vertex that forms a face $f=(u_k,w,v_k)$ in clockwise
  order. By the definition of layering, $w\in V_{k+1}\cup V_k\cup
  V_{k-1}$.

  \emph{Case~1(a).} $w\in V_{k+1}$: Then
  $S':=(u_0,u_1,\dots,u_{k-1},u_k,w,v_k,v_{k-1},\dots,v_1,v_0)$ is a
  lollipop. Since $f$ is a face, $R_B(S')=R_B(S)$ and
  $L(S')=L(S)-\{w\}$, contradicting (3).

  \emph{Case~1(b).} $w=u_{k-1}$: Observe that
  $S':=(u_0,u_1,\dots,u_{k-1},v_k,v_{k-1},\dots,v_1,v_0)$ is a
  lollipop of height $k-1$. We have $r_B(S')\leq
  r_B(S)+1<\frac{1}{3}|B|+1\leq \frac{2}{3}|B|$ (since $|B|\geq3$).
  Hence $S'$ satisfies (1).  If $u_k\in B$ then $r_B(S')>r_B(S)$,
  contradicting (2). Now assume that $u_k\not\in B$. We have
  $\ell(S')=\ell(S)$ and $r(S')=r(S)+1$, contradicting (4).

  \emph{Case~1(c).} $w=v_{k-1}$: This case is analogous to Case~1(b)
  except that we use
  $S':=(u_0,u_1,\dots,u_{k-1},u_k,v_{k-1},\dots,v_1,v_0)$.

  \emph{Case~1(d).} $w\in V_{k-1}\-\{u_{k-1},v_{k-1}\}$: There
  is a monotone path $P=(w=z_{k-1},z_{k-2},\dots,z_0=r)$ such that if
  some $z_i=u_i$ then $z_j=u_j$ for each $j\in[0,i]$, and if some
  $z_i=v_i$ then $z_j=v_j$ for each $j\in[0,i]$. Observe that both
  $S':=(u_0,u_1,\dots,u_{k-1},u_k,z_{k-1},z_{k-2},\dots,z_1,z_0)$ and
  $S'':=(z_0,z_1,\dots,z_{k-1},v_k,v_{k-1},\dots,v_1,v_0)$ are
  lollipops. By assumption,
  $r_B(\overleftarrow{S})=\ell_B(S)>\frac{2}{3}|B|$.  We have
  $r_B(S')>\frac{2}{3}|B|$ since $S'$ is a lollipop with $r_B(S')\geq
  r_B(S)$ and $\ell(S')<\ell(S)$.  Similarly,
  $r_B(S'')>\frac{2}{3}|B|$. Hence
  $r_B(\overleftarrow{S})+r_B(S')+r_B(S'')>2|B|$.  Thus
  $R_B(\overleftarrow{S})\cap R_B(S')\cap R_B(S'')\neq\emptyset$,
  which is a contradiction since $R_B(S')\cap R_B(S'')\subseteq
  R_B(S)=L_B(\overleftarrow{S})$.

  \emph{Case~1(e).} $w\in V_{k}$: This case is analogous to Case~1(d),
  except that here $P$ is a monotone path
  $(w=z_{k},z_{k-1},\dots,z_0=r)$, and
  $S':=(u_0,u_1,\dots,u_{k-1},u_k,z_{k},z_{k-1},\dots,z_1,z_0)$ and
  $S'':=(z_0,z_1,\dots,z_{k-1},z_{k},v_k,v_{k-1},\dots,v_1,v_0)$.

  \emph{Case~2.}
  $S=(u_0,u_1,\dots,u_{k-1},u_k,u_{k+1},v_k,v_{k-1},\dots,v_1,v_0)$:
  Let $w$ be the vertex that forms a face $f=(u_k,w,u_{k+1})$ in
  clockwise order. Hence $w\in V_{k+1} \cup V_k$.

  \emph{Case~2(a).} $w\in V_{k+1}$: This case is analogous to Case~1(a) except that 
  $S':=(u_0,u_1,\dots,u_{k-1},u_k,w,u_{k+1},v_k,v_{k-1},\dots,v_1,v_0)$.

  \emph{Case~2(b).} $w=v_k$: This case is analogous to Case~1(c)
  except with
  $S':=(u_0,u_1,\dots,u_{k-1},u_k,v_k,v_{k-1},\dots,v_1,v_0)$.

  \emph{Case~2(c).} $w\in V_{k}-\{v_k\}$: This case is
  analogous to Case~1(d), except that here $P$ is a monotone path
  $(w=z_{k},z_{k-1},\dots,z_0)$, and
  $S':=(u_0,u_1,\dots,u_{k-1},u_k,z_{k},z_{k-1},\dots,z_1,z_0)$ and
  $S'':=(z_0,z_1,\dots,z_{k-1},z_k,u_{k+1},v_k,v_{k-1},\dots,v_1,v_0)$.

  Each case leads to a contradiction. Hence $r_B(S)\leq\frac{2}{3}|B|$
  and $\ell_B(S)\leq\frac{2}{3}|B|$.  Let $G_1$ be the subgraph
  induced by the vertices in $S$ and to the right of $S$.  Let $G_2$
  be the subgraph induced by the vertices in $S$ and to the left of
  $S$. By the Jordan Curve Theorem, no vertex to the right of $S$ is
  adjacent to a vertex to the left of $S$. Hence $G=G_1\cup G_2$ and
  $(G_1,G_2)$ is the desired separation.
\end{proof}

Lemmas~\ref{General} and \ref{PlanarSeparation} together prove
Theorem~\ref{NonRepPlanar} (since every planar graph with at least
four vertices is a spanning
subgraph of a plane triangulation).

\section{Proof of Theorem~\ref{LocalPlanar}}

Theorem~\ref{LocalPlanar} is a special case of the following result
with $H=K_5$ or $H=K_{3,3}$. A graph $H$ is \emph{apex} if $H-v$ is planar for some
vertex $v$ of $H$.

\begin{theorem}
  \label{thm:ExcludeApex}
  For every fixed apex graph $H$ there is a constant $c=c(H)$ such
  that, for every integer $k$, every $H$-minor-free graph $G$ is
  $c^{k^2}$-colourable such that $G$ contains no repetitively coloured
  path of order at most $2k$.
\end{theorem}

\begin{proof}
  \citet{Eppstein-Algo00} proved that for some function $f$ (depending
  on $H$), for every $H$-minor-free graph $G$, for every vertex $r$ of
  $G$, and for every integer $\ell\geq0$, the set of vertices in $G$
  at distance at most $\ell$ from $r$ induces a subgraph of treewidth
  at most $f(\ell)$. This is called the \emph{diameter-treewidth} or
  \emph{bounded local treewidth} property; also see
  \citep{DH-SJDM04,DH-Algo04,Grohe-Comb03}.  \citet{DH-SODA04}
  strengthened Eppstein's result by showing that one can take
  $f(\ell)=c\ell$ for some constant $c=c(H)$.

  Let $G$ be an $H$-minor-free graph. By considering each connected
  component in turn, we may assume that $G$ is connected.  Let $r$ be
  a vertex of $G$.  Let $V_0,V_1,\dots,V_p$ be the layering of $G$
  starting at some vertex $r$ of $G$. Fix an integer $k\geq 1$.  For
  $i\in[1,p]$, let $G_i:=G[V_i\cup V_{i+1}\cup\dots\cup
  V_{\min\{p,i+2k-1\}}]$, and let $G'_i$ be the minor of $G$ obtained
  by contracting the connected subgraph $G[V_0\cup V_1\cup\dots\cup
  V_{i-1}]$ into a single vertex $r_i$. Thus $G_i'$ is an
  $H$-minor-free graph containing $G_i$ as a subgraph, and each vertex
  in $G_i$ is at distance at most $2k$ from $r_i$ in $G'_i$. By the
  diameter-treewidth property, $G_i$ has treewidth at most $2ck$. By a
  theorem of \citet{KP-DM08}, there is a nonrepetitive
  $4^{2ck}$-colouring $\psi_i$ of $G_i$.

  For each vertex $v$ of $G$, define
  $\psi(v):=(\phi_0(v),\phi_1(v),\dots,\phi_{2k-1}(v))$, where
  $\phi_j(v):=\psi_i(v)$ and $i$ is the unique integer for which
  $i\equiv j\pmod{2k}$ and $v\in V(G_i)$.  Suppose on the contrary
  that $G$ contains a repetitively coloured path
  $P=(v_1,\dots,v_{2t})$ of order at most $2k$ (under the colouring
  $\psi$).  Thus $P$ is contained in some $G_i$. Let $j:=i\bmod{2k}$.
  Hence $\psi_i(v_a)=\phi_j(v_a)=\phi_j(v_{t+a})=\psi_i(v_{t+a})$ for
  each $a\in[1,t]$. That is, $P$ is repetitively coloured by $\psi_i$
  in the colouring of $G_i$. This contradiction proves that $G$
  contains no repetitively coloured path under $\psi$. The number of
  colours is $(4^{2ck})^{2k}=(4^{4c})^{k^2}$.
\end{proof}

Graphs embeddable on a fixed surface exclude a fixed apex graph as a
minor \citep{Eppstein-Algo00}. Thus Theorem~\ref{thm:ExcludeApex}
implies:

\begin{corollary}
  For every fixed surface $\Sigma$ there is a constant $c=c(\Sigma)$
  such that, for every integer $k\geq 1$, every graph $G$ embeddable
  in $\Sigma$ is $c^{k^2}$-colourable such that $G$ contains no
  repetitively coloured path of order at most $2k$.
\end{corollary}

\section{Open Problems}

Our research suggests two open problems: \vspace*{-2ex}
\begin{enumerate}
\item Is $\pi(G)\in o(\log n)$ for every planar graph $G$ with $n$
  vertices?
\item Is there a polynomial function $f$ such that for every integer
  $k\geq 1$ every planar graph $G$ is $f(k)$-colourable such that $G$
  contains no repetitively coloured path of order at most $2k$?
\end{enumerate}

Finally, we mention a class of planar graphs that seem difficult to
nonrepetitively colour. Let $T$ be a tree rooted at a vertex $r$. Let
$V_i$ be the set of vertices in $T$ at distance $i$ from $r$.  Draw
$T$ in the plane with no crossings. Add a cycle on each $V_i$ in the
cyclic order defined by the drawing to create a planar graph $G_T$. It
is open whether $\pi(G_T)\leq c$ for some constant $c$ independent of
$T$. Note that this class of planar graphs includes examples with
unbounded degree and unbounded treewidth.


\def\soft#1{\leavevmode\setbox0=\hbox{h}\dimen7=\ht0\advance \dimen7
  by-1ex\relax\if t#1\relax\rlap{\raise.6\dimen7
  \hbox{\kern.3ex\char'47}}#1\relax\else\if T#1\relax
  \rlap{\raise.5\dimen7\hbox{\kern1.3ex\char'47}}#1\relax \else\if
  d#1\relax\rlap{\raise.5\dimen7\hbox{\kern.9ex \char'47}}#1\relax\else\if
  D#1\relax\rlap{\raise.5\dimen7 \hbox{\kern1.4ex\char'47}}#1\relax\else\if
  l#1\relax \rlap{\raise.5\dimen7\hbox{\kern.4ex\char'47}}#1\relax \else\if
  L#1\relax\rlap{\raise.5\dimen7\hbox{\kern.7ex
  \char'47}}#1\relax\else\message{accent \string\soft \space #1 not
  defined!}#1\relax\fi\fi\fi\fi\fi\fi}


\appendix\section{Lower Bounds}\label{LowerBound}

\citet{BV-NonRepVertex07} constructed a planar graph $G$ with
$\pi(G)\geq 10$. Pascal Ochem [private communication] observed that
this lower bound can be improved to $11$ by adapting a construction
due to \citet{Albertson-EJC04} as follows. \citet{BV-NonRepVertex07}
constructed an outerplanar graph $H$ with $\pi(H)\geq 7$. Let $G$ be
the following planar graph. Start with a path $P=(v_1,\dots,v_{22})$.
Add two adjacent vertices $x$ and $y$ that both dominate $P$. Let each
vertex $v_i$ in $P$ be adjacent to every vertex in a copy $H_i$ of
$H$. Suppose on the contrary that $G$ is nonrepetitively
$10$-colourable. Without loss of generality, $x$ and $y$ are
respectively coloured $1$ and $2$. A vertex in $P$ is \emph{redundant}
if its colour is used on some other vertex in $P$. If no two adjacent
vertices in $P$ are redundant then at least $11$ colours appear
exactly once on $P$, which is a contradiction. Thus some pair of
consecutive vertices $v_i$ and $v_{i+1}$ in $P$ are redundant. Without
loss of generality, $v_i$ and $v_{i+1}$ are respectively coloured $3$
and $4$. If some vertex in $H_i\cup H_{i+1}$ is coloured $1$ or $2$,
then since $v_i$ and $v_{i+1}$ are redundant, with $x$ or $y$ we have
a repetitively coloured path on 4 vertices. Now assume that no vertex
in $H_i\cup H_{i+1}$ is coloured $1$ or $2$. If some vertex in $H_i$
is coloured $4$ and some vertex in $H_{i+1}$ is coloured $3$, then
with $v_i$ and $v_{i+1}$, we have a repetitively coloured path on 4
vertices. Thus no vertex in $H_i$ is coloured $4$ or no vertex in
$H_{i+1}$ is coloured $3$. Without loss of generality, no vertex in
$H_i$ is coloured $4$. Since $v_i$ dominates $H_i$, no vertex in $H_i$
is coloured $3$. We have proved that no vertex in $H_i$ is coloured
$1,2,3$ or $4$, which is a contradiction, since $\pi(H_i)\geq
7$. Therefore $\pi(G)\geq 11$.

\end{document}